\documentclass[preprints,article,accept,moreauthors]{Definitions/mdpi} 
\usepackage{comment}
\usepackage{tikz}
\usepackage{placeins}

\newcommand{\R}{\mathbb{R}}

\newcommand{\bbm}{\begin{bmatrix}}
\newcommand{\ebm}{\end{bmatrix}}

\newcommand{\N}{\operatorname{\mathbb{N}}}

\newcommand{\OR}{\operatorname{\overline{\mathbb{R}}}}

\newcommand{\ratio}{\frac{\vert \xkone-\xstar \vert}{\vert \xk-\xstar \vert }}
\newcommand{\ratioquad}{\frac{\vert \xkone-\xstar \vert}{\vert \xk-\xstar \vert^2 }}
\newcommand{\ratioalpha}{\frac{\vert \xkone-\xstar \vert}{\vert \xk-\xstar \vert^\alpha}}
\newcommand{\xstar}{x_*}
\newcommand{\xk}{x_k}
\newcommand{\xkone}{x_{k+1}}
\newcommand{\seq}{\{ x_k \}}

\newcommand{\oqo}{\overline{Q}_1(x_k)}
\newcommand{\oqoy}{\overline{Q}_1(y_k)}

\newcommand{\ek}{e_k}

\newcommand{\x}{\mathbf{x}}

\firstpage{1} 
\pubvolume{1}
\issuenum{1}
\articlenumber{0}
\pubyear{2026}
\copyrightyear{2026}
\datereceived{ } 
\daterevised{ } % Comment out if no revised date
\dateaccepted{ } 
\datepublished{ } 
\Title{ A 
$p$-step generalization of the Q-order of convergence}

\Author{Gabriel Jarry-Bolduc$^{1}$\orcidA{}}

\AuthorNames{Firstname Lastname, Firstname Lastname and Firstname Lastname}

\address{%
$^{1}$ \quad Department of Mathematics and Statistics, American University of Sharjah, Sharjah, United Arab Emirates; +971 6 515 4423; gabjarry@alumni.ubc.ca}

\abstract{The notion of \emph{Q-order convergence} is arguably the most important tool for describing the asymptotic behavior  of a convergent sequence. Loosely speaking, it captures the``speed''of convergence of an iterative method.  The concept of Q-order convergence is not always well suited for sequences whose errors do not decrease monotonically at every step. In this paper, we introduce the notion of \emph{$p$-step Q-order convergence}. It generalizes the classical notion of Q-order convergence by comparing errors that are 
$p$ iterations apart rather than errors of successive iterates.  This definition recovers classical Q-order convergence as the special case
$p=1$. We show that it extracts meaningful convergence information from certain non-monotonic sequences for which the classical Q-order either does not exist or assigns an overly pessimistic classification. We develop the basic theory of the new notion and locate it within the classical hierarchy by proving that 
$p$-step Q-order at least
$\alpha$ implies \emph{R-order at least 
$\alpha$}. Natural applications include iterative methods whose updates alternate or cycle over multiple steps.}

\keyword{Q-order convergence; R-order convergence; Non-monotonic sequences.} 

\begin{document}

%%%%%%%%%%%%%%%%%%%%%%%%%%%%%%%%%%%%%%%%%%

\section{Introduction}\label{sec:intro}

Iterative methods are a fundamental aspect of numerical analysis and computational sciences. They provide procedures for solving equations, linear or non-linear systems, and optimization problems among others \cite{OrtegaRheinboldt2000,Kelley1995,Kelley1999}. A central question in studying any such method is how fast the iterative method converges. The concept of \emph{order of convergence} quantifies how rapidly the error decreases as the iterates approach a solution. It plays an important role in both the theoretical classification of methods and their practical performance. A rigorous convergence theory  makes it possible to compare algorithms, supports the derivation of error estimates, and clarifies the limits of what iterative procedures can achieve. Since the early work of Cauchy and the significant work later developed in  \cite{Traub1964,Ostrowski1973}, this topic  has  been extensively studied. Several definitions have been created to measure the  ``speed'' of convergence of an iterative method (see \cite{catinas2019survey,catinas2021} for a complete overview of the main definitions).

Arguably the most popular notion to describe the convergence  of a sequence is called  \emph{Q-order of convergence} \cite[Section 2.1.1]{catinas2019survey}. It provides information  on the asymptotic behavior of the ratio  between consecutive errors. The three main types of  $Q$-order convergence are called \emph{linear}, \emph{superlinear}, and \emph{quadratic} convergence. Essentially, quadratic convergence means that the number of correct digits approximately doubles at each iteration $k$  once $k$ is sufficiently large. One of the most famous results regarding  Q-order convergence is the quadratic convergence theorem for \emph{Newton's method}  rigorously proved in  \cite{Kantorovich1948}. A detailed study of these ideas, including the notions of Q-order and \emph{R-order convergence}, was later carried out in  \cite{OrtegaRheinboldt2000}. 
The R-order convergence measures the asymptotic rate of decay of the
error sequence as a whole; it is defined via roots of the errors rather
than quotients of consecutive errors.  It remains well defined
even when the ratios of consecutive errors fluctuate. It is usually preferred for non-monotonic sequences  since the Q-order convergence may not exist or may be misleading.  The relations  between these two notions  are summarized in  \cite{catinas2019survey,catinas2021}.

Important progress on the notion of order of convergence was made in \cite{Potra1989}. The author establishes necessary and sufficient conditions for the equivalence  between Q-order convergence and R-order  convergence.  The characterization of superlinear convergence is clarified in both settings. The concept of superlinear convergence has been defined in (at least) two different ways in the literature, which has led to some confusion, since the two definitions are not equivalent (see \cite[Section 2]{catinas2021} for more details). The question of how reliably one can estimate the order of convergence from a finite sequence of iterates was examined in~\cite{BeyerEbanksQualls1990,jay2001note}. Acceleration and extrapolation methods, such as the \emph{epsilon algorithm} \cite{Shanks1955,Wynn1956,Brezinski1977,Brezinski1980},
motivated notions of convergence beyond the classical Q- and R-order
framework.

%Fixed-point theory offers a natural setting for these questions. The \emph{Banach contraction principle} guarantees linear convergence for contractive mappings, while higher-order behavior arises when successive Fr\'echet derivatives of the iteration operator vanish at the fixed point---an idea developed in detail by Traub~\cite{Traub1964} and extended in operator settings by Argyros~\cite{Argyros2008}. Argyros and his collaborators~\cite{Argyros2008} significantly broadened classical convergence theory by weakening traditional Lipschitz  assumptions, introducing center-Lipschitz conditions, and analyzing inexact and perturbed iterations in Banach spaces. 

A parallel line of research concerns multipoint and composition methods.
The well-known Kung-Traub conjecture \cite{KungTraub1974} asserts that a multipoint method without memory for finding
a simple root of a scalar nonlinear equation cannot exceed order
$2^{d-1}$, where $d$ is the number of evaluations of the function and
its derivatives per step.  Methods attaining this bound are called
\emph{optimal}. Families of optimal and near-optimal high-order schemes
were constructed and analyzed in
\cite{SharmaGuha2007,CorderoTorregrosa2010,Petkovic2013}. The relationship between order, computational cost, and the structure of the
iteration has since become a central theme of the theory.

 In optimization, the convergence rates of \emph{steepest descent}, \emph{Newton}, and \emph{quasi-Newton} methods have been extensively analyzed \cite{Fletcher1987,nocedal2006numerical}.  For instance,  steepest descent converges linearly and can be extremely slow on \emph{ill-conditioned problems}. Incorporating exact or approximate second-order information into the search direction, as in Newton and quasi-Newton methods, improves the convergence rate from linear to quadratic or superlinear, respectively.

Despite this extensive development, certain structural limitations remain. The notion of Q-order convergence is defined through the quotients of
consecutive errors and is therefore ill suited to methods whose progress
is realized over blocks of several iterations. On the other hand, R-order convergence, while insensitive to step-to-step fluctuations, provides only a global bound on the error sequence and does not describe such blockwise behavior either. Multi-step solvers, block-iterative procedures, and schemes that alternate between distinct update rules often display periodic or blockwise reduction.  For instance, a simple convergent sequence defined differently on even and odd indices is given in \cite{jay2001note}. Its error increases at every other step, so the quotients of consecutive errors are unbounded.  Therefore, the sequence possesses no Q-order of convergence. This illustrates a structural limitation of the Q-order framework: a definition based on the ratio of two consecutive errors cannot adequately describe sequences whose errors decrease regularly over blocks of iterations but non-monotonically within them.

This paper proposes a notion of \emph{$p$-step Q-order of convergence},
in which convergence is measured by comparing errors $p$ iterations
apart rather than consecutively. The definition extends the classical Q-order concept and reduces to it when $p=1$. We develop the basic theory related to this notion and clarify its relationship to Q-order and R-order convergence. The $p$-step Q-order convergence  makes it possible to captures features of non-monotonic convergent sequences that the R-order cannot distinguish.  In doing so, we aim to provide a unified framework for analyzing iterative methods whose error reduction follows a multi-step pattern. 

In summary, the main contributions of this paper are the following:
\begin{itemize}
  \item We introduce the $p$-step Q-order of convergence 
  (Definition~\ref{def:pstepdef}).
  \item We characterize when the definition is satisfied.
  \item We show that lower orders are inherited and that validity for 
  step $p$ propagates to every multiple $mp$.
  \item We prove that $p$-step Q-order at least $\alpha$ implies R-order 
  at least $\alpha$ (Theorem~\ref{thm:pstep-implies-Rorder}). Examples 
  show that the new notion sits strictly between the Q-order and the R-order (Figure \ref{fig:hierarchy}).
\end{itemize}

This paper is organized as follows. In Section \ref{sec:prel}, the background results and definitions are introduced.  In Section \ref{sec:pstepQorder}, the notion of p-step Q-order convergence is defined. Theoretical results are provided and the relationship to Q-order convergence and R-order convergence is clarified. Last, the main results  of this paper are summarized and future research directions are proposed in Section \ref{sec:conclusion}.

%--------------------------------------------------------------------------------------
\section{Preliminaries}\label{sec:prel}

In this  paper, we follow the notation  used in the classical textbook Numerical Analysis by Burden \cite{Burden2016}. The set of natural numbers is denoted by $\N$ and equal to $\{1, 2, \dots \}.$  When $\{x_k\}$  is a sequence  of real numbers converging to $x_*$,  we often define   $e_k=\vert x_k-x_* \vert,$  and refer to $e_k$ as the error at iteration $k$.  The sequence $\{e_k\}$ is called  the sequence of errors.   The definition of an eventually decreasing sequence of real numbers follows.

\begin{Definition}[Eventually decreasing sequence]\label{def:eventDecreasing}
A sequence of real numbers $\{a_k\}_{k \ge 0}$ is \emph{eventually decreasing} 
if there exists $K \ge 0$ such that $a_{k+1} < a_k$ for all $k \ge K$.
\end{Definition}

Since the definition of Q-order convergence varies slightly across the
literature, we present three common versions and identify the one used
as the default throughout the paper.  Note that the definitions may be  extended to sequences of vectors  in $\R^n$ by taking the norm rather than the absolute value.  However, 
the choice of norm influences the value of the constant in the following definition, and particular care must be exercised if defining \emph{sublinear convergence}.

\begin{Definition}[Numerical Analysis textbook definition] \cite[Definition 2.7]{Burden2016} \label{def:burden}

Let  $\{ x_k\}$ be a sequence of real numbers that converges to $x_*$ with $\xk\neq \xstar.$ If there exist $c>0$ and $\alpha \geq 1$ such that   $$\lim_{k \to \infty} \ratioalpha=c,$$
then $\seq$ converges to $\xstar$ with order $\alpha$ and asymptotic error constant $c.$ 

\noindent If $\alpha=1$ and $c<1$, then the convergence  is said to be \emph{linear}.

\noindent If $\alpha=2,$ then the convergence is said to be \emph{quadratic}.  
\end{Definition}

Based on this definition,  if  $\alpha=1$ and $c=0,$ we may not  conclude that  the sequence is linearly   convergent since the constant $c$ is taken to be positive in the definition. 
Later in the textbook, superlinear convergence is defined as follows: if $\alpha=1$ and $c=0,$ then the sequence is said to be \emph{superlinear}. This is somewhat confusing since the definition assumes $c>0.$ Note  that Definition \ref{def:burden} is sometimes called \emph{C-order convergence}.  

Next, we provide the definitions of linear, superlinear and quadratic according to the classical textbook Numerical Optimization \cite{nocedal2006numerical}.

\begin{Definition}[Numerical Optimization textbook definition] \cite[Appendix A.2]{nocedal2006numerical} \label{def:nocedal}

    Let $\seq$ be a sequence of real numbers that converges to $\xstar.$ We say that the convergence is \emph{Q-linear} if  there is a constant $c_1 \in (0,1)$ such that $$\frac{\vert \xkone-\xstar \vert}{\vert \xk-\xstar \vert } \leq c_1, \quad \text{for all $k$ sufficiently large.}  $$
    The convergence is said to be \emph{$Q$-superlinear} if  $$\lim_{k \to \infty} \ratio=0.$$
    The convergence is said to be \emph{Q-quadratic} if  there exists a positive constant $c_2$  such that $$ \ratioquad \leq c_2, \quad \text{for all $k$ sufficiently large}.$$
\end{Definition}
The exact same definition is used in  several optimization textbooks such as \cite{andrei2022modern,dennisschnabel1983,Kelley1995}. Another almost identical definition is proposed in \cite[Section 3.3.1]{conn2000trust}. The only difference is that the definition is provided without forming a quotient which makes it possible for a convergent sequence to have $\vert \xk-\xstar \vert=0$ for all $k$ larger than some $K\geq 0.$ From Definition \ref{def:nocedal}, we get that quadratic convergence implies superlinear convergence, and superlinear convergence implies linear convergence. However, it is not necessarily the  case using Definition \ref{def:burden}. Indeed,  superlinear convergence does not imply linear  convergence.    

Finally,  the  definition that will be used in this paper is introduced. This definition applies to a wider class of sequences since a \emph{limit superior} is used.  Using a limit superior, the  ratios of consecutive errors always exists in  the extended  nonnegative real  numbers
$[0,+\infty]$, whereas the corresponding limit may fail to exist when these ratios oscillate.

\begin{Definition}[Default definition {\cite[Sect.~2.1.1]{catinas2019survey}}]
\label{def:catinas}
Let $\{\xk\}$ be a sequence converging to $\xstar$ with $\xk \neq \xstar$ for 
all large $k$, and define $\ek = \vert \xk - \xstar \vert$. For 
$\alpha \geq 1$, define the \emph{quotient convergence factors} and their supremum 
limits
\[
Q_\alpha(x_k) = \frac{e_{k+1}}{e_k^{\alpha}}, 
\qquad 
\overline{Q}_\alpha(x_k)= \limsup_{k \to \infty} Q_\alpha(x_k).
\]
We say that the sequence $\{\xk\}$
\begin{itemize}
    \item has \emph{no Q-order} if $\overline{Q}_1(x_k) =\infty$;
    \item converges \emph{Q-sublinearly} if $1 \leq \overline{Q}_1(x_k) <\infty$;
    \item converges \emph{at least Q-linearly} (\emph{at least Q-order $1$}) if 
          $\overline{Q}_1(x_k) < 1$;
    \item converges \emph{at least Q-superlinearly} (\emph{at least Q-superorder $1$}) if 
          $\overline{Q}_1(x_k) = 0$.
\end{itemize}
For $\alpha > 1$, we say that the sequence  $\{\xk\}$ 
\begin{itemize}
    \item converges with \emph{Q-order at least $\alpha$} if $\overline{Q}_\alpha(x_k) < \infty$;
    \item converges with \emph{Q-superorder at least $\alpha$} if $\overline{Q}_\alpha(x_k) = 0$.
\end{itemize}
\end{Definition}

 When $\alpha=2,$ the order is usually  called \emph{quadratic}. 
 
The limit superior in Definition~\ref{def:catinas} ensures that 
$\overline{Q}_\alpha(x_k) = \limsup_{k\to\infty} Q_\alpha(x_k)$ 
exists in the extended nonnegative real numbers. However, it does not 
resolve the lack of informativeness for non-monotonic sequences: 
whenever the error increases along a subsequence, 
$\overline{Q}_\alpha(x_k) = \infty$ for every $\alpha \ge 1$, 
and the definition assigns no Q-order to the sequence, regardless of 
its overall rate of convergence. In the remaining of this paper, Definition \ref{def:catinas} is used as the default  definition for Q-order convergence.

A weaker notion of convergence order is obtained by bounding the errors
by an auxiliary sequence of known order rather than by relating
consecutive errors directly.

\begin{Definition}[R-order of convergence {\cite[Definition 4.4]{jay2001note}}]
\label{def:rorder}
Let $\alpha \geq 1$. A sequence $\{\xk\}$ converges to $\xstar$ with 
\emph{R-order at least $\alpha$} if there exists a sequence $\{\beta_k\}$ of 
positive real numbers with $\ek \leq \beta_k$ for all sufficiently large 
$k$, converging to zero with Q-order at least $\alpha$ in the sense of 
Definition~\ref{def:catinas}.
\end{Definition}

Unlike the Q-order, which constrains the ratio of consecutive errors at every step, the R-order of convergence requires only that the errors be dominated by an auxiliary sequence converging with the corresponding Q-order. It is therefore insensitive to stepwise fluctuations. For the same reason, it retains no information about the actual step-to-step behavior of the errors. We are now ready to introduce the novel notion of $p$-step Q-order convergence.

%-------------------------------------------------------------------------

\section{Main Results} \label{sec:pstepQorder}

This section introduces the notion of $p$-step Q-order convergence, 
in which the ratios of errors are taken over $p$ iterations for some
$p \in \N$, rather than a single iteration as in 
Definition~\ref{def:catinas}. As we will see, this flexibility makes 
it possible to extract meaningful convergence information for certain 
non-monotonic sequences. We first provide some observations on Q-order 
convergence and decreasing sequences of errors.

In Definition \ref{def:catinas}, note that  $\overline{Q}_1(x_k)<1$ implies that the sequence of errors is  eventually decreasing. The converse is not true. For instance, consider the sequence  defined  by $x_k = \tfrac{1}{k}$, which converges 
to $x^* = 0$ with errors $e_k = \tfrac{1}{k}$. The correct statement is that if the sequence of errors $\{\ek\}$ is  eventually decreasing, then $\oqo\leq 1.$  
When $\oqo=1,$ then the sequence of errors may or may not be eventually decreasing.  For instance, consider the sequence $x_k=\frac{1}{k}+\frac{(-1)^k}{k^2}.$ Then $\oqo=1,$  and the sequence of errors $e_k=x_k$ is not eventually decreasing.

It is a necessary condition for a sequence of errors to be eventually decreasing to have convergence  at least Q-linear.  If $\overline{Q_1}(x_k)>1,$  then it implies  that the sequence of errors is not eventually decreasing. Moreover,  we get that $\overline{Q}_1(x_k)=\infty$  if and only if the the ratio of errors is unbounded.  

To summarize, we have the following results regarding not eventually decreasing sequence of errors $\{e_k\}$.

\begin{Proposition} \label{prop:nondecreasing}
    Let $\xk$  be a sequence that converges to $\xstar$ with $\xk \neq \xstar$ for  large $k$. Define $\ek=\vert x_k -\xstar \vert.$ If $\{\ek\}$ is not eventually decreasing, then $x_k$ converges Q-sublinearly or has no Q-order.
\end{Proposition}
\begin{proof}
Since $\{\ek\}$ is not eventually decreasing, there exist infinitely many 
indices $k_i, i \in \N,$ such that $e_{k_i+1} \geq e_{k_i}$. This means
$Q_1(x_{k_i}) \geq 1$ for all $i$. We get
\[
\overline{Q}_1(x_k) \;\geq\; \limsup_{i\to\infty} Q_1(x_{k_i}) \;\geq\; 1.
\]
So the convergence is not at least Q-linear.  Clearly, for any 
$\alpha > 1$,
$\overline{Q}_\alpha(x_k) = \infty$ (since Q-order at least $\alpha$ for $\alpha>1$ implies Q-order at least $\tilde{\alpha}$ for all $\tilde{\alpha} \in [1, \alpha)$). Therefore,  the sequence $\{\xk\}$ converges Q-sublinearly or has no Q-order.
\end{proof}

As noted by several authors, this can be  misleading as illustrated in the next example. 

\begin{Example} \label{ex:jaysExample}
Consider the sequence $\{y_k\}_{k=0}^\infty$ given in \cite{jay2001note}, defined by
\[
y_k = 
\begin{cases}
2^{-2^{k}}, & k \text{ even},\\[2pt]
2^{-3^{k}}, & k \text{ odd}.
\end{cases}
\]
The sequence converges to $y_* = 0$, so that $e_k = y_k$. We get
\[
Q_1(y_{2k+1}) = \frac{y_{2k+2}}{y_{2k+1}} = 2^{\,3^{2k+1} - 2^{2k+2}} 
\longrightarrow \infty  \quad \text{as} \quad k \to \infty.
\]
Hence $\oqoy = \infty$. By Definition~\ref{def:catinas}, the sequence 
has no Q-order.
\end{Example}

However, observe that  $\{y_k\}$ converges to $\xstar = 0$ faster than $x_k = 2^{-k}$ 
in the sense that $y_k < x_k$ for all $k \geq 0$.  The sequence $\{x_k\}$ converges 
at least Q-linearly.  This illustrates that the 
absence of a Q-order for $\{y_k\}$ not a statement about slow convergence in this case, but about 
irregular stepwise behavior.

Proposition \ref{prop:nondecreasing} and Example \ref{ex:jaysExample} show 
that all   definitions  of  Q-order  are  not well-suited
for error sequences that are not eventually decreasing. It either assigns 
no order or reduces to Q-sublinearity, regardless of how fast the sequence 
converges. The following definition is designed to capture the convergence 
behavior of such sequences.

\begin{Definition}[$p$-step Q-order of convergence]\label{def:pstepdef}
Let $\xk$ be a sequence converging to $\xstar$ with $\xk \neq \xstar$ for 
all large $k$, and define $\ek= \vert \xk - \xstar \vert$. For $p \in \N$ 
and $\alpha \geq 1$,  define the \emph{$p$-step quotient convergence factors} by
\[
Q_\alpha^{(p)}(x_k) = 
\begin{cases}
\left( \dfrac{e_{k+p}}{e_k} \right)^{\!1/p} 
    & \text{if } \alpha = 1, \\[10pt]
\left( \dfrac{e_{k+p}}{e_k^{\alpha^p}} \right)^{\!\frac{\alpha-1}{\alpha^p-1}} 
    & \text{if } \alpha > 1,
\end{cases}
\]
and define
\[
\overline{Q}_\alpha^{(p)}(x_k) = \limsup_{k \to \infty} Q_\alpha^{(p)}(x_k).
\]
 We say that the sequence $\{\xk\}$
\begin{itemize}
    \item has no \emph{$p$-step Q-order} if $\overline{Q}_1^{(p)}(x_k) = \infty$;
    \item has \emph{no Q-order of any step length} if $\overline{Q}_1^{(p)}(x_k) = \infty$ for any $p \in \N.$
    \item converges \emph{$p$-step Q-sublinearly} if 
          $1 \leq \overline{Q}_1^{(p)}(x_k) <\infty $;
    \item converges  \emph{$p$-step at least Q-linearly} (\emph{$p$-step at least Q-order 1}) if 
          $\overline{Q}_1^{(p)}(x_k) < 1$;
    \item converges \emph{$p$-step at least Q-superlinearly} (\emph{$p$-step at least Q-superorder 1}) if 
          $\overline{Q}_1^{(p)}(x_k) = 0$.
\end{itemize}
For $\alpha > 1$, we say that $\{\xk\}$ 
\begin{itemize}
    \item converges with \emph{$p$-step  Q-order at least $\alpha$} if 
          $\overline{Q}_\alpha^{(p)}(x_k) < \infty$;
    \item converges  with \emph{$p$-step at least Q-superorder at least  $\alpha$} if 
          $\overline{Q}_\alpha^{(p)}(x_k) = 0$.
\end{itemize}
\end{Definition}

The definition is designed so that $\alpha$ retains its classical meaning 
as an average convergence order per iteration, regardless of the step 
$p$. Indeed, if a sequence satisfies  $e_{k+1} = C e_k^{\alpha}$ 
with $\alpha > 1$ and $C > 0$, then composing $p$ steps gives
\[
e_{k+p} = C^{\,1 + \alpha + \cdots + \alpha^{p-1}} e_k^{\alpha^{p}} 
        = C^{\frac{\alpha^{p}-1}{\alpha-1}}\, e_k^{\alpha^{p}}.
\]
So the exponent of  the error $e_k$ accumulated over a block of $p$ iterations is 
$\alpha^{p}$, and the constant $C$ accumulates the geometric sum 
$\tfrac{\alpha^{p}-1}{\alpha-1}$. The $p$-step quotient convergence factor $Q_\alpha^{(p)}(x_k)$ 
compares $e_{k+p}$ against $e_k^{\alpha^{p}},$ and extracts the root 
$\tfrac{\alpha-1}{\alpha^{p}-1}$, so that $Q_\alpha^{(p)}(x_k) = C$ for 
any $p \in \N$. This means that two $p$-step quotient convergence factors with different values of  $p$ are measured on a 
common per-iteration scale and may be compared directly. The case 
$\alpha = 1$ follows the same idea.   Note that 
$\tfrac{\alpha-1}{\alpha^{p}-1} \to \tfrac{1}{p}$ as $\alpha \to 1^{+}$. As a function of $\alpha,$ the $p$-step quotient convergence factor $Q_\alpha^{(p)}(x_k)$ is a continuous function.  Additionally, note that when  $p = 1,$  both cases in  $Q_\alpha^{(p)}(x_k)$
 reduce to the quotient convergence factor $Q_\alpha(x_k)$ in  Definition \ref{def:catinas}.

Two more properties  of the definition are worth emphasizing. First, the 
limit superior ranges over all indices $k$, not over a subsequence. 
The resulting notion is a uniform, quotient-type condition in the spirit 
of the classical Q-order, rather than a statement about a chosen 
subsequence. The accumulated reduction over every 
$p$ consecutive iterations is controlled. For instance,  if a sequence  is 3-step Q-order at least 2,  then   every 3 iterations   the numbers of correct digits is multiplied by approximately $2^3$ (at least) once $k$ is sufficiently large. Second, the classification is 
stated for each fixed $p \in \N$ separately (except for  no Q-order of any step length). For instance, a sequence may have 
 no $1$-step Q-order while converging with $2$-step Q-order at least 
$2$, as Example~\ref{ex:jaysExample} showed.  Note that it is possible for a sequence to have no Q-order of any step length.   The following example illustrates this situation.
\begin{Example} \label{ex:squares}
Consider the sequence $\seq_{k=0}^\infty$ defined by
\begin{equation*}
\xk=\begin{cases}
     &10^{-2^k}, \quad \text{if $k$ is not a perfect square,} \\
    &10^{-k}, \quad \text{if $k$ is a perfect square.}
\end{cases}
\end{equation*}
Then $\{x_k\}$ has no $Q$-order of any step length. Indeed, No fixed block length can absorb the slow terms $10^{-k}$, due to the fact that the gaps between 
consecutive perfect squares grow without bound.   For every $p \in \N$, 
infinitely many windows of length $p$ end at a perfect square. Fix 
$p \in \N$ and let $k = m^2-p$. For all large $m$, the index $k$ is not a 
perfect square, so
\[
Q_1^{(p)}(x_{m^2-p}) 
= \left( \frac{10^{-m^2}}{10^{-2^{m^2 - p}}} \right)^{\!1/p} 
= 10^{\frac{2^{m^2-p} - m^2}{p}} \longrightarrow \infty 
\quad \text{as} \quad m \to \infty.
\]
Therefore,
$\overline{Q}_1^{(p)}(x_k) = \infty$ for every $p \in \N$, and 
$\{x_k\}$ has no Q-order of any step length.
\end{Example}

One may wonder if a definition with the step size given as  a function of $k$ could be developed to deal with the previous example. Although theoretically possible, allowing
$p$ to depend on 
$k$ would reduce the definition to an order statement about a chosen subsequence of iterates.  Keeping 
$p$ constant makes the 
$p$-step Q-order a uniform condition: every window of 
$p$ consecutive iterations is controlled. This preserves the interpretation of
$\alpha$ as an (average) order  of convergence per iteration.

We now clarify that if a sequence is $p$ in Q-order at least $\alpha$  for some $\alpha>1$ and $p \in \N,$ then the sequence is $p$ in Q-order at least $\tilde{\alpha}$  for all $\tilde{\alpha} \in [1, \alpha).$ This is not surprising since a similar statement can be made for the definition of Q-order convergence (Definition \ref{def:catinas}).

\begin{Proposition}\label{prop:monotoneAlpha}
Let $\{\xk\}$ converge to $\xstar$ with $\xk \neq \xstar$ for all large 
$k$ and let $p \in \N$. Suppose $\{\xk\}$ converges with $p$-step Q-order 
at least $\alpha$ for some $\alpha > 1$. Then for every 
$\tilde{\alpha} \in [1, \alpha)$, the sequence $\seq$ converges with $p$-step 
Q-superorder at least $\tilde{\alpha}$. In particular, it converges with 
$p$-step Q-order at least $\tilde{\alpha}$.
\end{Proposition}

\begin{proof}
Set $M= \overline{Q}_\alpha^{(p)}(x_k) < \infty$. Then 
$Q_\alpha^{(p)}(x_k) \leq M + 1$ for all $k$ sufficiently large.  
Raising this inequality to the power $\tfrac{\alpha^{p}-1}{\alpha-1} > 0$ 
gives
\[
e_{k+p} \leq C\, e_k^{\,\alpha^{p}}, 
\quad \text{where} \quad  C= (M+1)^{\frac{\alpha^{p}-1}{\alpha-1}},
\]
for all sufficiently large $k$. Let $\tilde{\alpha} \in [1, \alpha)$. Then
\[
\frac{e_{k+p}}{e_k^{\,\tilde{\alpha}^{p}}} 
\;\leq\; C\, e_k^{\,\alpha^{p} - \tilde{\alpha}^{p}} 
\longrightarrow 0 \quad   \text{as}  \quad k \to \infty,
\]
since $\alpha^{p}-\tilde{\alpha}^{p}>0$ and $e_k \to 0$. By Definition~\ref{def:pstepdef}, $Q_{\tilde{\alpha}}^{(p)}(x_k)$ is the 
quotient $e_{k+p}/e_k^{\tilde{\alpha}^{p}}$ raised to the positive power $\tfrac{\tilde{\alpha}-1}{\tilde{\alpha}^{p}-1}$ if 
$\tilde{\alpha} > 1$, and $\tfrac{1}{p}$ if $\tilde{\alpha} = 1$, in which 
case $\tilde{\alpha}^{p} = 1$. Since this quotient tends to $0$, 
$Q_{\tilde{\alpha}}^{(p)}(x_k) \to 0$. Therefore, 
$\overline{Q}_{\tilde{\alpha}}^{(p)}(x_k) = 0$.
\end{proof}

We now clarify conditions for Definition \ref{def:pstepdef}to be satisfied for some $p$.  We say that Definition \ref{def:pstepdef} is \emph{satisfied} if the sequence converges with $p$-step at Q-order at least $\alpha$ for some $\alpha \geq 1$ and for some $p \in \N.$

\begin{Proposition}\label{prop:satisfied}
Let $\{\xk\}$ be a sequence converging to $\xstar$ with $\xk \neq \xstar$ 
for all large $k$. Then Definition~\ref{def:pstepdef} is satisfied for some 
$p \in \N$ if and only if there exist $c \in (0,1)$ and $K \in \N$ such that
\begin{equation}\label{eq:contraction}
    e_{k+p} \leq c^{\,p}\, e_k \quad \text{for all } k \geq K.
\end{equation}
\end{Proposition}

\begin{proof}
  Suppose Definition~\ref{def:pstepdef} is satisfied. By Proposition \ref{prop:monotoneAlpha}, we may assume $\alpha=1$. That is   
$\overline{Q}_1^{(p)}(x_k) < 1$ for some $p \in \N$. Choose 
$c \in \bigl(\overline{Q}_1^{(p)}(x_k),\, 1\bigr)$. By the definition of 
the limit superior, $Q_1^{(p)}(x_k) \leq c$ for all $k$ sufficiently large. Therefore, we obtain 
Equation \eqref{eq:contraction}. Conversely, Equation \eqref{eq:contraction} yields 
$Q_1^{(p)}(x_k) \leq c$ for all $k \geq K$.  It follows that  
$\overline{Q}_1^{(p)}(x_k) \leq c < 1$ and Definition~\ref{def:pstepdef} is 
satisfied.
\end{proof}

\begin{Proposition}\label{prop:nonmonotone-conditions}
Let $\{\xk\}$ converge to $\xstar$ with $\xk \neq \xstar$ for all large $k$, 
and let $p \in \N$. Define $\ek=\vert x_k -\xstar \vert.$ Then
\begin{equation}\label{eq:interleaved}
\overline{Q}_1^{(p)}(x_k) 
= \max_{0 \leq j \leq p-1} \limsup_{i \to \infty} 
\left( \frac{e_{j+(i+1)p}}{e_{j+ip}} \right)^{\!1/p}.
\end{equation}
Consequently, Definition~\ref{def:pstepdef} is satisfied for $p$ if and only 
if there exists $c \in (0,1)$ such that each subsequence 
$\{e_{j+ip}\}_{i=0}^\infty$, $j = 0, 1,  \dots, p-1$, eventually satisfies 
$e_{j+(i+1)p} \leq c^{p}\, e_{j+ip}$. \\ In particular, if 
Definition~\ref{def:pstepdef} is satisfied for $p$, then  $\{e_k\}$ eventually admits no run of $p$ consecutive non-decreasing 
steps.
\end{Proposition}

\begin{proof}
Let $K \in \N$ be such that $x_k \neq \xstar$ for all $k \geq K$, so
that $Q_1^{(p)}(x_k)$ is well defined for all $k \geq K$. Writing
$k = j + ip$ with $j \in \{0, 1, \dots, p-1\}$, we have
$Q_1^{(p)}(x_k) = \bigl(e_{j+(i+1)p}/e_{j+ip}\bigr)^{1/p}$. So the $p$
residue classes modulo $p$ partition $\{Q_1^{(p)}(x_k)\}_{k \geq K}$
into finitely many subsequences. Since the limit superior of a sequence
equals the largest of the limit superiors of the subsequences in such a
partition, Equation \eqref{eq:interleaved} follows.

Suppose Definition \ref{def:pstepdef} is satisfied. That is  $\overline{Q}_1^{(p)}(x_k) < 1.$ Pick
$c \in \bigl(\overline{Q}_1^{(p)}(x_k), 1\bigr)$. By
Equation \eqref{eq:interleaved}, each of the $p$ limit superiors is less than
$c$, so for each $j,$ there exists   $I_j$ such that
$e_{j+(i+1)p} \leq c^p e_{j+ip}$ for all $i \geq I_j$. Conversely, if
such a $c \in (0,1)$ exists, then each limit superior in
Equation \eqref{eq:interleaved} is at most $c$. Therefore,
$\overline{Q}_1^{(p)}(x_k) \leq c < 1$.

Finally, suppose Definition~\ref{def:pstepdef} is satisfied for $p$.
Set $K'= \max_{0 \leq j \leq p-1} (j + I_j p)$. Then
$e_{k+p} \leq c^p\, e_k < e_k$ for all $k \geq K'$. Additionally, if 
$e_k \leq e_{k+1} \leq \cdots \leq e_{k+p}$ for some $k \geq K'$, then
$e_{k+p} \geq e_k$, a contradiction.
\end{proof}

Proposition~\ref{prop:nonmonotone-conditions} shows that the $p$-step 
Q-order does not monitor the error sequence step by step.  It monitors the 
$p$ interleaved subsequences $\{e_{j+ip}\}_{i=0}^\infty$, $j = 0, 1, \dots, p-1$. 
Definition~\ref{def:pstepdef} is satisfied precisely when each of these 
subsequences contracts geometrically at a common per-iteration rate. What 
happens between consecutive terms of a subsequence is irrelevant: a single 
step may increase the error by an arbitrarily large factor, provided every 
error remains below a fixed fraction of the error $p$ iterations earlier. 
Therefore, the definition  tolerates non-monotone behavior confined to a 
window of length $p$, but not non-monotone behavior that persists across 
windows. We saw that Example \ref{ex:jaysExample} satisfies 
Definition \ref{def:pstepdef} with $p = 2$ despite oscillating, because its 
increases never last more than one step. On the other hand,  in 
Example \ref{ex:squares}, the non-monotonicity outruns every possible fixed window, 
so the definition fails for all $p$. The next result shows that if Definition~\ref{def:pstepdef} is satisfied for some step $p$, then it is satisfied for every positive integer
multiple of $p$.

\begin{Proposition}\label{prop:multiples}
Let $\{\xk\}$ converge to $\xstar$ with $\xk \neq \xstar$ for all large $k$. 
If Definition~\ref{def:pstepdef} is satisfied for some step $p \in \N$ and order $\alpha \geq 1$, then it 
is satisfied with order $\alpha$  and any step  $mp$ where $m \in \N$.
\end{Proposition}

\begin{proof}
(\emph{Case $\alpha = 1$.}) By
Proposition~\ref{prop:nonmonotone-conditions}, there exists
$c \in (0,1)$ such that $e_{k+p} \leq c^{p}\, e_k$ for all large $k$.
Chaining this bound $m$ times gives the inequality
$e_{k+mp} \leq c^{mp}\, e_k$ for all large $k$.  By
Proposition~\ref{prop:nonmonotone-conditions},
Definition~\ref{def:pstepdef} is satisfied  with step $mp$ and $\alpha=1$.

(\emph{Case $\alpha > 1$.}) Since
$\overline{Q}_\alpha^{(p)}(x_k) < \infty$, there exists $C \geq 1$ such
that $e_{k+p} \leq C\, e_k^{\alpha^{p}}$ for all large $k$. Setting
$b_k = C^{1/(\alpha^{p}-1)} e_k$, this bound is now
$b_{k+p} \leq b_k^{\alpha^{p}}$.  Chaining $m$ times gives
$b_{k+mp} \leq b_k^{\alpha^{mp}}$.  That is,
$e_{k+mp} \leq C^{s}\, e_k^{\alpha^{mp}}$ with
$s = \tfrac{\alpha^{mp}-1}{\alpha^{p}-1}$. Hence
$\overline{Q}_\alpha^{(mp)}(x_k) \leq C^{s} < \infty$. 
Therefore, Definition~\ref{def:pstepdef} is satisfied  with step $mp$ and $\alpha \geq 1$.
\end{proof}

Note that the converse is not true.  It is not necessarily true that if the definition of $p$-step Q-order convergence is satisfied with step 6 and order $\alpha \geq 1$, then it must be  satisfied for step 3 and the same order $\alpha \geq 1.$

Lastly, we clarify the relation between $p$-step Q-order convergence and R-order convergence.

\begin{Theorem}\label{thm:pstep-implies-Rorder}
Let $\{\xk\}$ converge to $\xstar$ with $\xk \neq \xstar$ for all large $k$. 
Let $p \in \N$ and $\alpha \geq 1$. If $\{\xk\}$ has $p$-step Q-order at 
least $\alpha$, then $\{\xk\}$ has R-order at least $\alpha$.
\end{Theorem}

\begin{proof}
(\emph{Case $\alpha = 1$.}) By
Proposition \ref{prop:nonmonotone-conditions}, there exist
$c \in (0,1)$ and $K \in \N$ such that 
\begin{equation}\label{eq:boundChain}
e_{k+p} \leq c^{p} e_k \quad  \text{for all
$k \geq K$.}
\end{equation}
For $k \geq K$, write $k = K + j + ip$ with $0 \leq j \leq p-1$.
Chaining the inequality in \eqref{eq:boundChain} $i$ times gives
\begin{equation}\label{eq:mprime}
e_k \leq c^{ip} e_{K+j} \leq M c^{ip} \leq M' c^{k} \quad \text{where} \quad 
M= \max_{0 \leq j \leq p-1} e_{K+j} \quad \text{and} \quad 
M' = M c^{-(K+p-1)}.
\end{equation}
The last inequality in \eqref{eq:mprime} uses
$ip \geq k - K - p + 1$ and $c < 1$. Observe that the sequence
$\beta_k= M' c^{k}$ is positive, converges to $0$ with Q-order $1$,
and satisfies $e_k \leq \beta_k$ for all $k \geq K$.  By
Definition \ref{def:rorder}, $\{\xk\}$ has R-order at least $1$.

\noindent (\emph{Case $\alpha > 1$.}) Since
$\overline{Q}_\alpha^{(p)}(x_k) < \infty$, there exists $C \geq 1$ such
that 
\begin{equation}\label{eq:powerbound}
b_{k+p} \leq b_k^{\alpha^{p}} \quad \text{for all large $k$ and where} \quad  b_k= C^{1/(\alpha^{p}-1)} e_k.
\end{equation}
As $b_k \to 0$,
there exist $K \in \N$ and $\lambda \in (0,1)$ such that both the inequality in \eqref{eq:powerbound}
 and $b_k \leq \lambda$ hold for all $k \geq K$. For $k \geq K$,
write $k = K + j + ip$ with $0 \leq j \leq p-1$.  Chaining the inequality in \eqref{eq:powerbound}
 $i$ times gives
\begin{equation} \label{eq:lambda}
b_k \leq b_{K+j}^{\alpha^{ip}} \leq \lambda^{\alpha^{ip}}
\leq \gamma^{\alpha^{k}} \quad  \text{where} 
\quad \gamma= \lambda^{\alpha^{-(K+p-1)}} \in (0,1).
\end{equation}  
The last inequality in \eqref{eq:lambda}
uses $ip \geq k - K - p + 1$ and $\lambda < 1$. Since $C \geq 1$, we 
obtain 
\begin{equation}
e_k \leq b_k \leq \gamma^{\alpha^{k}} \quad  \text{for all $k \geq K$.} 
\end{equation}
The sequence $\beta_k= \gamma^{\alpha^{k}}$ is positive and satisfies
$\beta_{k+1} = \beta_k^{\alpha}$.  Hence, it converges to $0$ with
Q-order at least $\alpha$.  By Definition~\ref{def:rorder}, $\{\xk\}$
has R-order at least $\alpha$.
\end{proof}

The converse of Theorem \ref{thm:pstep-implies-Rorder} does not necessarily hold.  The sequence 
in Example \ref{ex:squares} has R-order at least $1$, but has no Q-order for any step length. Together with 
Example \ref{ex:jaysExample}, this places the $p$-step Q-order convergence strictly 
between the Q-order convergence and the R-order convergence. Figure \ref{fig:hierarchy} illustrates the three different cnvergence notions.  

\begin{figure}[h]
\centering
\begin{tikzpicture}[every node/.style={font=\small}]
  \definecolor{setgray}{HTML}{F1EFE8}
  \definecolor{setgrayline}{HTML}{5F5E5A}
  \definecolor{setteal}{HTML}{E1F5EE}
  \definecolor{settealline}{HTML}{0F6E56}
  \definecolor{setpurple}{HTML}{EEEDFE}
  \definecolor{setpurpleline}{HTML}{534AB7}
  \filldraw[fill=setgray,   draw=setgrayline,   thick] (0,0)     ellipse (5.9 and 3.5);
  \filldraw[fill=setteal,   draw=settealline,   thick] (0,-0.35) ellipse (4.6 and 2.6);
  \filldraw[fill=setpurple, draw=setpurpleline, thick] (0,-0.75) ellipse (2.3 and 1.35);
  \node at (0, 2.95) {R-order at least $\alpha$};
  \node at (0, 1.6)  {$p$-step Q-order at least $\alpha$, for some $p$};
  \node at (0,-0.6)  {Q-order at least $\alpha$};
  \node at (0,-1.05) {($p=1$)};
  \fill[settealline] (-1.8, 1.0) circle (1.6pt) 
    node[right=2pt, text=black] {Example~\ref{ex:jaysExample} ($p=2$)};
  \fill[setgrayline] (-2.6, 2.55) circle (1.6pt) 
    node[right=2pt, text=black] {Example~\ref{ex:squares}};
\end{tikzpicture}
\caption{The $p$-step Q-order convergence sits strictly between the Q-order and the 
R-order. Example \ref{ex:jaysExample} has $2$-step Q-order at least $2$ but no 
Q-order, and Example~\ref{ex:squares} has R-order at least $1$ but  has no Q-order for any step length.}
\label{fig:hierarchy}
\end{figure}
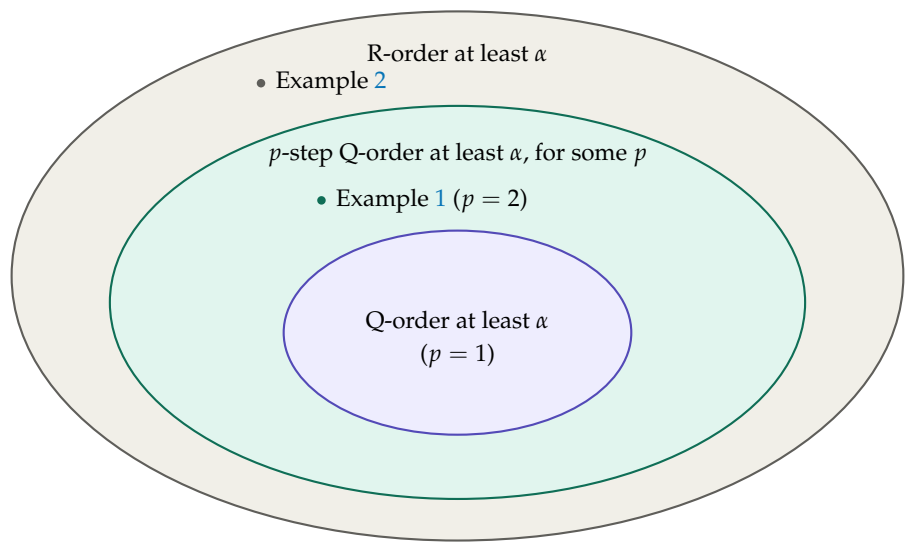

\FloatBarrier

\section{Conclusion} \label{sec:conclusion}
%%%%%%%%%%%%%%%%%%%%%%%%%%%%%%%%%%%%%%%%%%

In this paper, we introduced the notion of $p$-step Q-order convergence, 
which measures error reduction over blocks of $p$ consecutive iterations. The notion  
reduces to the classical Q-order definition when $p = 1$. The ``normalization'' of 
the $p$-step quotient convergence factors ensures that $\alpha$ retains 
its meaning as a per-iteration order, so that factors associated with 
different block lengths are measured on a common scale. We characterized 
when the definition is satisfied for a given $p$: a geometric contraction 
over $p$ steps, or equivalently the geometric contraction of the $p$ 
interleaved subsequences $\{e_{j+ip}\}_i$ at a common rate. Similar  to the Q-order convergence definition,  We showed 
that lower orders of $\alpha$ are inherited.  Besides,  we showed that  if $p$-step Q-order convergence  is satisfied for some $p$, then it must be satisfied 
for every positive integer multiple  of $p$ as the step.  However, the converse does not necessarily hold.  Another of the main results is that $p$-step Q-order at 
least $\alpha$ implies R-order at least $\alpha$. Examples \ref{ex:jaysExample} and \ref{ex:squares} 
demonstrate that both containments in Figure \ref{fig:hierarchy} are strict: the $p$-step Q-order captures sequences whose non-monotone 
behavior is confined to a fixed window, while sequences with 
irregularities separated by gaps that grow without bound may not satisfy the definition of 
$p$-step Q-order convergence for any $p$, as Example \ref{ex:squares} shows.

Several directions merit further investigation. A natural next step is to 
apply the framework to concrete iterative methods whose update rules 
alternate or cycle over multiple steps, such as multi-step solvers, 
block-iterative procedures, and optimization methods with non-monotone 
line searches.  Second, the question of estimating the pair $(p, \alpha)$ 
numerically from a finite sequence of iterates in the spirit of the 
computational convergence orders surveyed in \cite{catinas2019survey}, deserves attention. Extending the theory to sequences of vectors, where the choice 
of norm may affect the definition of sublinear could be investigated.  Finally, studying the smallest  valid step $p$ for  a given sequence is  left as future work.

%%%%%%%%%%%%%%%%%%%%%%%%%%%%%%%%%%%%%%%%%%
%\section{Patents}

%This section is not mandatory but may be added if there are patents resulting from the work reported in this manuscript.

%%%%%%%%%%%%%%%%%%%%%%%%%%%%%%%%%%%%%%%%%%
\vspace{6pt} 

%%%%%%%%%%%%%%%%%%%%%%%%%%%%%%%%%%%%%%%%%%
%% optional
%\supplementary{The following supporting information can be downloaded at \linksupplementary{s1}, Figure S1: title; Table S1: title; Video S1: title.}

% Only for journal Methods and Protocols:
% If you wish to submit a video article, please do so with any other supplementary material.
% \supplementary{The following supporting information can be downloaded at \linksupplementary{s1}, Figure S1: title; Table S1: title; Video S1: title. A supporting video article is available at doi: link.}

% Only used for preprtints:
% \supplementary{The following supporting information can be downloaded at the website of this paper posted on \href{https://www.preprints.org/}{Preprints.org}.}

% Only for journal Hardware:
% If you wish to submit a video article, please do so with any other supplementary material.
% \supplementary{The following supporting information can be downloaded at \linksupplementary{s1}, Figure S1: title; Table S1: title; Video S1: title.\vspace{6pt}\\
%\begin{tabularx}{\textwidth}{lll}
%\toprule
%\textbf{Name} & \textbf{Type} & \textbf{Description} \\
%\midrule
%S1 & Python script (.py) & Script of python source code used in XX \\
%S2 & Text (.txt) & Script of modelling code used to make Figure X \\
%S3 & Text (.txt) & Raw data from experiment X \\
%S4 & Video (.mp4) & Video demonstrating the hardware in use \\
%... & ... & ... \\
%\bottomrule
%\end{tabularx}
%}

%%%%%%%%%%%%%%%%%%%%%%%%%%%%%%%%%%%%%%%%%%

\funding{This research was funded by  the American University of Sharjah, Faculty Research Grant  FRG26-S27.}

\institutionalreview{Not applicable}

\informedconsent{Not applicable}

\dataavailability{No new data were created or analyzed in this study. Data sharing is not applicable to this article.}

\acknowledgments{During the preparation of this manuscript, the author used Claude Fable 5 for the purposes of rephrasing some sentences, proofreading, and to generate Figure \ref{fig:hierarchy}. The author have reviewed and edited the output and take full responsibility for the content of this publication.}

\conflictsofinterest{The authors declare no conflicts of interest. The funders had no role in the design of the study; in the collection, analyses, or interpretation of data; in the writing of the manuscript; or in the decision to publish the results.} 

%%%%%%%%%%%%%%%%%%%%%%%%%%%%%%%%%%%%%%%%%%

%%%%%%%%%%%%%%%%%%%%%%%%%%%%%%%%%%%%%%%%%%
%\isPreprints{}{% This command is only used for ``preprints''.
%\begin{adjustwidth}{-\extralength}{0cm}
%} % If the paper is ``preprints'', please uncomment this parenthesis.
%\printendnotes[custom] % Un-comment to print a list of endnotes

\reftitle{References}

% References must be numbered in order of appearance in the text (including citations in tables and legends) and listed individually at the end of the manuscript. We recommend preparing the references with a bibliography software package, such as EndNote, ReferenceManager or Zotero, to avoid typing mistakes and duplicated references. Include the digital object identifier (DOI) for all references where available.
% Please provide either the correct journal abbreviation (e.g. according to the “The list of Title Word Abbreviations” https://portal.issn.org/ltwa) or the full name of the journal. 
% Citations and references in the Supplementary Materials are permitted provided that they also appear in the reference list here. 
% In the text, reference numbers should be placed in square brackets [ ] and placed before the punctuation; for example, [1], [1–3] or [1,3]. For embedded citations in the text with pagination, use both parentheses and brackets to indicate the reference number and page numbers; for example, [5] (p. 10), or [6] (pp. 101–105).

%=====================================
% References, variant A: external bibliography
%=====================================
% \bibliography{your_external_BibTeX_file}

%=====================================
% References, variant B: internal bibliography
%=====================================

% ACS format

%\bibliographystyle{siam}
\bibliography{bibliography}

% If authors have biography, please use the format below
%\section*{Short Biography of Authors}
%\bio
%{\raisebox{-0.35cm}{\includegraphics[width=3.5cm,height=5.3cm,clip,keepaspectratio]{Definitions/author1.pdf}}}
%{\textbf{Firstname Lastname} Biography of first author}
%
%\bio
%{\raisebox{-0.35cm}{\includegraphics[width=3.5cm,height=5.3cm,clip,keepaspectratio]{Definitions/author2.jpg}}}
%{\textbf{Firstname Lastname} Biography of second author}

% For the MDPI journals use author-date citation, please follow the formatting guidelines on http://www.mdpi.com/authors/references
% To cite two works by the same author: \citeauthor{ref-journal-1a} (\citeyear{ref-journal-1a}, \citeyear{ref-journal-1b}). This produces: Whittaker (1967, 1975)
% To cite two works by the same author with specific pages: \citeauthor{ref-journal-3a} (\citeyear{ref-journal-3a}, p. 328; \citeyear{ref-journal-3b}, p.475). This produces: Wong (1999, p. 328; 2000, p. 475)

%%%%%%%%%%%%%%%%%%%%%%%%%%%%%%%%%%%%%%%%%%
%% for journal Sci
%\reviewreports{\\
%Reviewer 1 comments and authors’ response\\
%Reviewer 2 comments and authors’ response\\
%Reviewer 3 comments and authors’ response
%}
%%%%%%%%%%%%%%%%%%%%%%%%%%%%%%%%%%%%%%%%%%
\PublishersNote{}
%\isPreprints{}{% This command is only used for ``preprints''.
%\end{adjustwidth}
%} % If the paper is ``preprints'', please uncomment this parenthesis.
\end{document}